\documentclass[11pt,a4paper]{article}

\usepackage{amsmath,amsthm,amsopn,amsfonts,amssymb,dsfont,color}
\usepackage{enumerate}
\usepackage{ucs}
\usepackage[utf8x]{inputenc}
\usepackage{graphicx}
\usepackage{ifthen}
\usepackage{subfigure}
\usepackage{epic}
\graphicspath{{eps/}{pdf/}}
\usepackage{authblk}
\usepackage{textcomp}
\usepackage{titlesec}

\usepackage{a4,a4wide}

\def\ep{{\varepsilon}}
\def\R{\mathbb R}

\def\({\left(}
\def\){\right)}

\newtheorem{theo}{\textbf{Theorem}}[section]

\newtheorem{prop}[theo]{\textbf{Proposition}}

\newtheorem{cor}[theo]{\textbf{Corollary}}

\newtheorem{assumption}[theo]{\textbf{Assumption}}
\numberwithin{equation}{section}

\title{Quantitative estimates of the threshold phenomena for propagation in reaction-diffusion equations}
\date{}
\author{}

\begin{document}

\maketitle

\begin{center}
{\large\bf Matthieu Alfaro \footnote{IMAG, Univ. Montpellier, CNRS, Montpellier, France.}, Arnaud Ducrot \footnote{Normandie Univ, UNIHAVRE, LMAH, FR-CNRS-3335, ISCN, 76600 Le Havre, France.}, Gr\'egory Faye \footnote{IMT, UMR 5219, Universit\'e de Toulouse, UPS-IMT, F-31062 Toulouse Cedex 9, France.}}\\
[2ex]

\end{center}

%\vspace{15pt}

%\tableofcontents

\vspace{10pt}

\begin{abstract}  We focus on the (sharp) threshold phenomena arising in some reaction-diffusion equations supplemented with some compactly supported initial data. In the so-called ignition and bistable cases, we prove the first sharp quantitative estimate on the (sharp) threshold values. Furthermore, numerical explorations allow to conjecture some refined estimates. Last we provide related results in the case of a degenerate monostable nonlinearity \lq\lq not enjoying the hair trigger effect\rq\rq.
\\


\end{abstract}

\section{Introduction}\label{s:intro}

In this work we consider the solution $u=u(t,x)$ of the reaction-diffusion equation
\begin{equation}\label{eq}
u_t=\Delta u+f(u), \quad t>0,\; x\in \R^N,
\end{equation}
supplemented with some radially symmetric compactly supported initial data. Typically, $f$ is a so-called ignition or bistable nonlinearity. It is well established \cite{flores}, \cite{Zla-06} that, roughly speaking, \lq\lq small'' initial data lead to extinction, whereas \lq\lq large'' initial data 
lead to propagation, which is referred as a {\it threshold phenomenon}. It was more recently proved \cite{Du-Mat-10}, \cite{Pol-11}, \cite{Mur-Zho-13, Mur-Zho-17} that, in many situations, a {\it sharp} threshold phenomenon does occur. Here, we say that there is a sharp threshold behavior when, for any strictly increasing family of initial data exhibiting extinction for sufficiently small values of the parameter and propagation for sufficiently large values of the parameter, there is exactly one member of the family for which neither extinction nor propagation occurs. For example, for smooth unbalanced bistable nonlinearity, in one space dimension, it is well known that at the threshold value, the corresponding solution of \eqref{eq} converges to the unique ground state centered at the origin \cite{flores}, \cite{Zla-06}, \cite{Du-Mat-10}, \cite{Pol-11}, \cite{Mur-Zho-13,Mur-Zho-17}. By a ground state of this equation we mean a positive stationary solution that decays to zero at infinity, that is a solution of
\begin{equation*}
0=\Delta u+f(u),\; x\in \R^N, \; u>0, \; \underset{|x|\rightarrow+\infty}{\lim}u(x)=0.
\end{equation*} 
As far as we know, no quantitative estimate of this (sharp) threshold phenomenon exists in the literature. The goal of the present work is to fill this gap by providing such estimates in some asymptotic regimes.

\medskip

Initiation of propagating fronts or pulses has a long history with multiple applications in both physical and natural sciences, e.g. flame combustion, epidemic outbreaks, ecological or bacterial invasion. We can in particular mention applications in neuroscience where one is typically interested to quantify the depolarization effect of local stimulation of an electrode on a nerve fiber \cite{Terman}, \cite{McK-Moll}, \cite{Neu-etal-97} or to characterize the critical stimulus amplitude and timing needed to generate waves in visual cortex \cite{FK18}.  Several approaches have been developed to derive criteria for initiation of propagation in bistable or excitable medium in one space dimension. For McKean-type nonlinearity $f(u)=-u+H(u-a)$, $a\in(0,1/2)$, and $H$ the Heaviside step function, an associated free boundary problem was studied by looking at the evolution of $m(t)=\sup\left\{ x>0 ~|~ u(t,x)=a\right\}$. By characterizing the asymptotic limiting behavior of $m(t)$, it was possible to obtain sharp threshold of propagation \cite{Terman}, \cite{McK-Moll} for some classes of initial conditions which cross the threshold $u=a$ only twice, and we refer to \cite{Moll-Ros} for a numerical treatment of the problem. More precisely, there is a trichotomy. If $m(t)$ is well defined for all time and $\underset{t\rightarrow+\infty}{\limsup}~m(t)=+\infty$, then there is propagation. If there exists $T>0$ such that $\underset{t\rightarrow T}{\liminf}~m(t)=0$, then there is extinction. Finally, if $m(t)$ remains uniformly bounded for all times, then the corresponding solution $u(t,x)$ converges to the unique ground state to \eqref{eq}. Another method to derive quantitative conditions for initiation of propagation in the one-dimensional case was obtained by rewriting the reaction-diffusion equation \eqref{eq} as a gradient-flow of the energy
\begin{equation}\label{energy}
\mathcal{E}(u)=\int \left[\frac 1 2|\nabla u|^2 + F(u) \right]dx,
\end{equation}
where $F(u)=-\int_0^u f(v)dv$, and projecting this gradient flow onto an approximate solution space \cite{Neu-etal-97}. A common choice for the approximation space consists of the amplitude and width of Gaussian profiles. This projects the infinite dimensional dynamical systems \eqref{eq} into a two-dimensional space and the criterion for initiating propagation takes the form of separatrices given by the stable manifold of the ground state which is a saddle node in this two-dimensional projected space. The idea that the ground state's stable manifold forms a threshold surface separating initial conditions leading to propagation or extinction was  used \cite{IB08} to derive estimates of the sharp threshold by 
approximating the stable manifold by its tangent linear space. Let us finally remark that the free boundary approach was recently applied to neural field equation with Heaviside step nonlinearity where it is possible to obtain explicit formula for the sharp threshold \cite{FK18} in this nonlocal setting.

Previous works on (sharp) threshold phenomena for \eqref{eq} rely on different tools such as the so-called {\it zero number argument} \cite{Du-Mat-10}, the combination of parabolic Liouville theorems and results on exponential separation and principal Floquet bundle \cite{Pol-11}, or energy methods \cite{Mur-Zho-13, Mur-Zho-17}, \cite{Alf-Duc-18}. Let us mention that the work \cite{Alf-Duc-18}, primarly concerned with a population dynamics model \cite{Kanarek-Webb} for {\it evolutionary rescue}, shows in particular that the problem 
$$
u_t=\Delta u +u(u-\theta(t))(1-u)
$$
where $\theta (t)\searrow 0$ as $t\to +\infty$, enjoys the so-called {\it hair trigger effect}: for any nonnegative and nontrivial initial data, the solution tends to 1 as $t\to +\infty$, locally uniformly in $x\in \R^{N}$. Since quantitative estimates for the bistable sharp threshold were missing, the main assumption was $\int _0^{+\infty}\theta(s)ds<+\infty$, which may be \lq\lq not optimal''.

Our approach is rather more direct, in the spirit of \cite{Zla-06}, and consists in constructing new and elaborate sub-and supersolutions that enable to capture the underlying time-space scaling of the problem. The sharp quantitative estimates we obtain may lead to refine existing results, such as those in \cite{Alf-Duc-18} for instance.

\medskip

We state below our assumptions and main results.

\begin{assumption}[Nonlinearity $f$]\label{hyp:nonlinearite}
The function $f:\R\to\R$ is locally Lipschitz continuous. There is a threshold $\theta\in(0,1)$ such that
\begin{equation}\label{f-zero}
f(u)=0\text{ for all }u\in (-\infty,0]\cup\{\theta\}\cup [1,\infty),
\end{equation}
and
\begin{equation}
\label{f-signe}
f(u)>0,\;\forall u\in (\theta,1),\quad \text{ and }\quad  \begin{cases}f(u)< 0,\; \forall u\in(0,\theta),\; \text{(BISTABLE CASE)}\\ \text{ or } \\ f(u)=0, \; \forall u\in(0,\theta),\; \text{ (IGNITION CASE).} \end{cases}
\end{equation}
In the bistable case, we further require
\begin{equation}
\int_0^1 f(s)ds>0.
\end{equation}
Moreover, the following  structure conditions hold.
\begin{itemize} 
\item[(i)]There are $r^+>0$ and $\delta^+\in (\theta,1)$ such that
\begin{equation*}
f(u)\leq r^+ (u-\theta),\;\forall u\in [\theta,\delta^+].
\end{equation*}
\item[(ii)] There are $r^->0$ and $\delta^-\in (\theta,1)$ such that
\begin{equation*}
f(u)\geq r^-\left(u-\theta\right),\;\forall u\in (-\infty,\delta^-].
\end{equation*}
\end{itemize}
\end{assumption}
 
Notice that the usual cubic bistable nonlinearity 
$$
f(u)=ru(u-\theta)(1-u),\;\forall u\in [0,1],
$$ 
where $r>0$, satisfies the above set of assumptions  as soon as $\theta<\frac{1}{2}$. Also,  an ignition linearity  satisfies the above items $(i)$ and $(ii)$ as soon as
$$
0<\liminf_{u\to \theta^+} \frac{f(u)}{u-\theta}\leq \limsup_{u\to \theta^+} \frac{f(u)}{u-\theta}<+\infty.
$$

Now, for $\ep\in (0,1-\theta)$ and $L>0$,  we consider the family of initial data $\phi_L^\ep$  given by
\begin{equation}\label{initial-data}
\phi_L^\ep(x)=(\theta+\ep) \mathbf 1_{B_L}(x),\;x\in\R^{N},
\end{equation}
wherein $ \mathbf 1 _A$ denotes the characteristic function of the set $A$, and $B_L$ the open ball of center $0\in \R^{N}$, and radius $L$. We denote by $u_L^\ep$ the solution of \eqref{eq} starting from the initial data $\phi_L^\ep$. Then this family of solutions enjoys the so-called threshold property.

\begin{prop}[Threshold property]\label{prop:threshold}
Let Assumption \ref{hyp:nonlinearite} hold. Let $\ep\in (0,1-\theta)$ be given. Then there exist $\widehat{L}_\ep>0$ and $\widetilde L_\ep>0$ such that
\begin{equation*}
\lim_{t\to+\infty} u_L^\ep (t,x)= \begin{cases} 
0\; \text{ uniformly in $\R^{N}$} \; &\text{ if }\; 0<L<\widehat{L}_\ep,\\
1\; \text{ locally uniformly in $\R^{N}$} \; &\text{ if }\; L>\widetilde{L}_\ep.
\end{cases}
\end{equation*}
\end{prop}

In a one dimensional framework, such threshold results were initiated by the seminal works of Kanel' \cite{Kan-64}, Aronson and Weinberger \cite{Aro-Wei-75}, Fife and McLeod \cite{Fif-Leo-77}. We also refer to Du and Matano \cite{Du-Mat-10} for more recent developments.

In arbitrary dimension and in our context, the existence of $\widehat L_\ep$ for the ignition case can, for instance, by found in \cite{Aro-Wei-78} or \cite{Mur-Zho-17}.  Note that the study of the ignition case is sufficient to conclude to the existence of $\widehat L_\ep$ for the bistable case due to the comparison principle. Notice also that the existence (as well as some estimates) of such small radii leading to extinction will also be proved in this work  for monostable nonlinearities (see Theorem \ref{th:extinction-deg-mono}), and thus for the ignition case due to comparison arguments.

In arbitrary dimension, the proof of the existence of $\widetilde L_\ep$ can, for instance, be found in \cite[Lemma 3.5]{Pol-11} in the case where $f$ is bistable, of the class $C^{1}$, and satisfies $f'(0)<0$. In our  context, since $f$ is Lipschitz continuous on $[0,1]$, we can use a small $C^1-$modification from below of the nonlinearity $f$ and construct a suitable subsolution converging to $1$ when $L$ is large enough. Hence the existence of $\widetilde L_\ep$ follows from the comparison principle for the bistable case, and thus for the ignition case due to comparison arguments.

\medskip

Now, for each $\ep\in (0,1-\theta)$, we consider the quantities $0<L_\ep^{ext}\leq L_\ep^{prop}$ given by
\begin{equation*}
\begin{split}
&\displaystyle L_\ep^{ext}:=\sup\left\{L>0:\;\;\lim_{t\to\infty} u_L^\ep(t,\cdot)=0\text{ uniformly in $\R^N$ }\right\},\\
&\displaystyle L_\ep^{prop}:=\inf\left\{L>0:\;\;\lim_{t\to\infty} u_L^\ep(t,\cdot)=1\text{ in }C^0_{\rm loc}(\R^N)\right\}.
\end{split}
\end{equation*}

In this work we derive sharp estimates for the above quantities in the asymptotic $\ep\to 0$.
We roughly prove that
$$
L_\ep^{ext} \text{ and } L_\ep^{prop} \approx \ln \frac{1}{\ep}\text{ as }\ep\ll 1.
$$
Our precise result reads as follows.

\begin{theo}[Quantitative estimates of the threshold]\label{th:main}
Let Assumption \ref{hyp:nonlinearite} hold. Then there are two constants $0<C^-<C^+$ such that
\begin{equation*}
C^-\leq \liminf_{\ep\to 0^+}\frac{L_\ep^{ext}}{\ln\frac{1}{\ep}}\leq \limsup_{\ep\to 0^+}\frac{L_\ep^{prop}}{\ln\frac{1}{\ep}}\leq C^+.
\end{equation*}
Moreover, the constants $C^\pm$ can be chosen as 
$$
C^-=\frac{1}{\sqrt{r^+}}, \quad C^+=\frac{2}{\sqrt{r^-}},
$$
where $r^\pm$ are as in Assumption \ref{hyp:nonlinearite}, items $(i)$ and $(ii)$.
\end{theo}

When $L_\ep^{ext}=L_\ep^{prop}:=L_\ep^{\star}$ for all small enough $\ep>0$ , we say that the threshold is sharp. Theorem \ref{th:main} immediately yields the following corollary.

\begin{cor}[Sharp threshold]\label{cor:sharp}
Let Assumption \ref{hyp:nonlinearite} hold. If the threshold is sharp then
\begin{equation*}
\frac{1}{\sqrt{r^+}}\leq \liminf_{\ep\to 0^+}\frac{L_\ep^\star}{\ln\frac{1}{\ep}}\leq \limsup_{\ep\to 0^+}\frac{L_\ep^\star}{\ln\frac{1}{\ep}}\leq \frac{2}{\sqrt{r^-}}.
\end{equation*}
\end{cor}

Let us recall that the sharp threshold phenomena was first analysed by Zlat\v{o}s \cite{Zla-06} in the one dimensional setting. 

In the bistable case, the threshold is known to be sharp when we further assume that $f$ is of the class $C^1$ and $f'(0)<0$, see \cite{Pol-11}, which is the case of the usual cubic nonlinearity \eqref{bistable}. We also refer to \cite[Theorem 4]{Mur-Zho-17} for other conditions insuring that the threshold is sharp in the bistable case.

In the ignition case, when $N=1$, the threshold is known to be sharp when  $f$ is nondecreasing in some right neighborhood of $\theta$, see \cite[Theorem 1.4]{Du-Mat-10}, which is the case of the nonlinearity \eqref{ignition}. When $N\geq 2$, we refer to  \cite[Theorem 5, Theorem 6]{Mur-Zho-17} for conditions ensuring that the threshold is sharp in the ignition case.

We now provide a corollary about the existence and the behavior as $\ep\to 0$ for the sharp threshold for the two following prototype nonlinearities
 \begin{equation}\label{ignition}
\text{Ignition: }f(u)=r(u-\theta)(1-u)\mathbf{1}_{(\theta,1)}(u),
\end{equation}
and
\begin{equation}\label{bistable}
\text{Bistable: }f(u)=ru(u-\theta)(1-u)\mathbf{1}_{(0,1)}(u),
\end{equation}
we shall use for numerical validations in Section \ref{s:numeric}. Here $r>0$ and $\theta\in (0,1)$ are given parameters with $\theta<1/2$ in the bistable case \eqref{bistable}.
In these two prototype situations, our main results above rewrite as follows.

\begin{cor}[Sharp threshold for prototype nonlinearities \eqref{ignition} and \eqref{bistable}]\label{cor:prototype}
Let $\ep\in (0,1-\theta)$ be given. We consider problem \eqref{eq} with the nonlinearities \eqref{ignition} and \eqref{bistable}. Then the family of initial data $\phi_L^\ep$ defined in \eqref{initial-data} exhibits a sharp threshold $L_\ep^\star$ that satisfies
$$
C^-\leq\liminf_{\ep\to 0}\frac{L_\ep^\star}{\ln\frac{1}{\ep}}\leq\limsup_{\ep\to 0} \frac{L_\ep^\star}{\ln\frac{1}{\ep}}\leq C^+,
$$
where the constants $C^\pm$ read as
$$
C^{-}=\frac{1}{\sqrt{r(1-\theta)}} ,\quad C^{+}=\frac{2}{\sqrt{r(1-\theta)}}\; \text{ for \eqref{ignition}},
$$
while
$$
C^{-}=\frac{1}{\sqrt{r\theta (1-\theta)}} ,\quad C^{+}=\frac{2}{\sqrt{r\theta (1-\theta)}}\; \text{ for \eqref{bistable}}.
$$
\end{cor}

\medskip

The organization of the paper is as follows. In Section \ref{s:toy-extinction} and Section \ref{s:toy-propagation} we inquire on extinction and non extinction phenomena in some related toy models. Next, in Section \ref{s:sharp}, we build on these preliminary results to prove Theorem \ref{th:main} and Corollary \ref{cor:prototype}, thus providing a quantitative estimate of the threshold phenomena. Section \ref{s:numeric} is devoted to numerical explorations that not only validate Theorem \ref{th:main} and Corollary \ref{cor:prototype} but also enable to make some conjectures on the best constants $C^\pm$. Last, in Section \ref{s:degenerate}, we present some  results on the threshold phenomena in the case of a degenerate monostable nonlinearity \lq\lq not enjoying the so-called hair trigger effect'', the typical example being $f(u)=ru^{p}(1-u)$, where $p>p_F:=1+\frac 2 N$. These results are not sharp with respect to $\ep\ll 1$ as in Theorem \ref{th:main} above. It 
however reflects the main order for the length of the sharp threshold for small $\ep$ and we believe our construction of adequate sub and supersolutions is instructive.

\section{Extinction in a toy model}\label{s:toy-extinction}

In this section, we consider the piecewise linear function
\begin{equation*}
g(u)=\left(u-\theta\right)_+.
\end{equation*}
Herein $\theta>0$ is a given and fixed parameter and the subscript $+$ is used to denote the positive part of a real number.
Let $N\geq 1$ be a given integer.
Consider the solution $w=w(t,x)$ of  the semi-linear problem
\begin{equation}\label{EQ}
w_t=\Delta w+g(w),\;t>0,\;x\in\R^N,
\end{equation} 
together with the initial data
\begin{equation}\label{EQ2}
w(0,x)=(\theta+\varepsilon) \mathbf 1_{B_L}(x),\;x\in\R^N,
\end{equation}
wherein $\ep>0$, $L>0$, and $B_L\subset \R^N$ denotes the ball of radius $L$ centred at the origin.

\begin{prop}[Extinction]\label{prop:extinction-toy}
Let $\delta>\theta$ be given. Consider the time 
\begin{equation}
\label{def:Tep}
T_\varepsilon:=\ln \frac{\delta-\theta}{\varepsilon}.
\end{equation}
Then, for any $0<\gamma<1$, there exists $\varepsilon_0>0$ small enough such that, for each $\ep\in (0,\ep_0)$ and for each $0<L<\gamma \ln\frac{1}{\ep}$, the  solution $w=w(t,x)$ of \eqref{EQ}---\eqref{EQ2} satisfies
\begin{equation}
\sup_{x\in\R^N}w\left(T_\varepsilon,x\right)\leq \theta,
\end{equation}
and is thus going to extinction at large times.
\end{prop}  

\begin{proof}
Consider the solution $v=v(t,x)$ of the heat equation
\begin{equation*}
v_t=\Delta v,\quad t>0,\;x\in\R^{N},
\end{equation*}
starting from $w(0,\cdot)$. Denote by $\Gamma=\Gamma(t,x)$ the heat kernel on $\R^N$, given by 
\begin{equation*}
\Gamma(t,x)=\frac{1}{(4\pi t)^{\frac N 2}}\exp\left(-\frac{|x|^2}{4t}\right),\;t>0,\;x\in\R^N,
\end{equation*}
where $|\cdot|$ is used to denote the Euclidean norm in $\R^N$.
Then $v(t,\cdot)=(\theta+\ep)\Gamma(t,\cdot) * \mathbf 1 _{B_L}$,  so that for all $t>0$
\begin{equation*}
V(t):=\Vert  v(t,\cdot)\Vert _{L^\infty(\R^N)}=\frac{\theta+\ep}{(4\pi t)^{\frac N 2}}\int_{\vert x\vert <L} e^{-\frac{\vert x\vert ^2}{4t}}dx=\frac{\theta+\ep}{(4\pi)^{\frac N 2}}\int_{\vert x\vert <\frac{L}{\sqrt t}} e^{-\frac{\vert x\vert ^2}{4}}dx.
\end{equation*}
We now construct a supersolution to \eqref{EQ} in the form $W(t,x):=v(t,x)\varphi(t)$, with $\varphi(0)=1$ and $\varphi(t)>0$. From $W_t-\Delta W=v\varphi'$ and the expression of  $g$ this requires
$$
\varphi'(t)\geq \left(\varphi(t)-\frac{\theta}{v(t,x)}\right)_+, \quad \forall (t,x)\in (0,+\infty)\times \R^N,
$$
and thus
\begin{equation*}
\varphi'(t)\geq \left(\varphi(t)-\frac{\theta}{V(t)}\right)_+, \quad \forall t\in (0,+\infty).
\end{equation*}
We now choose $\varphi$ as the solution of the Cauchy problem
$$
\varphi'(t)=\varphi(t)-\frac{\theta}{V(t)}, \quad \varphi(0)=1,
$$
that is 
$$
\varphi(t)=e^{t}\left(1-\int _0^{t}e^{-s}\frac{\theta}{V(s)}ds\right).
$$
Observe that $V(0)\varphi(0)>\theta$ and denote by $T>0$ the first time where $V(T)\varphi(T)=\theta$ (obviously we let $T=+\infty$ if such a time does not exist). Then $\left(\varphi(t)-\frac{\theta}{V(t)}\right)_+= \varphi(t)-\frac{\theta}{V(t)}$ for all $t\in [0,T)$, and thus $W(t,x)=v(t,x)\varphi(t)$ is a supersolution to \eqref{EQ} on the time interval $(0,T)$. In particular, if $T<+\infty$, it follows from the comparison principle that $w(T,\cdot)\leq W(T,\cdot)\leq \theta$, and we are done provided $T\leq T_\ep$, a condition we aim at reaching below.

Now, observe that the condition $T<+\infty$ rewrites as there exists $t>0$ satisfying the equation
\begin{equation*}
F_L(t):=\theta\left(1-\frac{e^{-t}}{A_L(t)}\right)+\ep-\int _0^{t}e^{-s}\frac{\theta}{A_L(s)}ds=0,
\end{equation*}
wherein we have set
$$
A_L(t):=\frac{1}{(4\pi)^{\frac N 2}}\int_{\vert x\vert <\frac{L}{\sqrt t}} e^{-\frac{\vert x\vert ^2}{4}}dx=C_N\int _0 ^ {\frac{L}{2\sqrt t}}r^{N-1} e^{-r^2}dr,
$$
for some constant $C_N>0$.
Since $F_L(0)=\ep$ and $F_L'(t)=\theta e^{-t}\frac{A_L'(t)}{A_L^{2}(t)}<0$, the condition $T\leq T_\ep$ is equivalent to require $
F_L(T_\varepsilon)<0$, that also reads as
\begin{equation}\label{condition}
1+\frac{\ep}{\theta}< \int _0^{T_\varepsilon}\frac{e^{-s}}{A_L(s)}ds+\frac{e^{-T_\varepsilon}}{A_L(T_\varepsilon)}.
\end{equation}

Note that the right-hand side of the above expression is decreasing with respect to $L$.
We now select
\begin{equation}
\label{select-Lep}
L_\varepsilon=\gamma \sqrt{T_\varepsilon \ln\frac{1}{\varepsilon}},
\end{equation}
for some constant $\gamma>0$ to be chosen later, and aim at proving that \eqref{condition}, with $L=L_\ep$, is satisfied for $\ep>0$ small enough. To do so, set
$$
G_\varepsilon:=\int _0^{T_\varepsilon}\frac{e^{-s}}{A_{L_\ep}(s)}ds.
$$ 
Integrating by parts yields
\begin{equation*}
G_\ep=\int _0^{T_\ep}\frac{e^{-s}}{A_{L_\ep}(s)}ds=1-\frac{e^{-T_\ep}}{A_{L_\ep}(T_\ep)}+\frac{C_N}{2^{N+1}}L_\ep^N\int_0^{T_\ep} \frac{e^{-s}}{A_{L_\ep}^{2}(s)} \frac 1{s^{\frac{N+2}{2}}}e^{-\frac{L_\ep^2}{4s}}ds.
\end{equation*}
Next the change of variable $z=s/L_\ep^2$ yields
\begin{equation*}
G_\ep=1-\frac{e^{-T_\ep}}{A_{L_\ep}(T_\ep)}+\frac{C_N}{2^{N+1}}\int_0^{\frac{1}{\gamma^2\ln\frac{1}{\ep}}} \frac{e^{-L_\ep^2z}}{A_1^{2}(z)}\frac 1{z^{\frac{N+2}{2}}}
e^{-\frac{1}{4z}}dz.
\end{equation*}
Hence \eqref{condition} rewrites
\begin{equation}\label{condition2}
\frac{2^{N+1}}{C_N\theta} \ep <H_\ep:=\int_0^{\frac{1}{\gamma^2\ln\frac{1}{\ep}}} \frac{e^{-L_\ep^2z-\frac{1}{4z}}}{A_1^{2}(z)}\frac 1{z^{\frac{N+2}{2}}}dz.
\end{equation}
Observe that  the function $z\in(0,+\infty) \mapsto L_\ep^2z+\frac{1}{4z}$ is decreasing then increasing and reaches its minimal value $L_\ep$  at $z=\frac{1}{2L_\ep}$.
Next one has, as $\ep \to 0$,   
\begin{equation*}
2L_\ep =2\gamma \sqrt{\ln (\delta-\theta)\ln\frac{1}{\ep}+\ln^2\frac{1}{\varepsilon}}=2\gamma\ln\frac{1}{\ep}\left(1+\mathcal O\left(\ln^{-1}\frac{1}{\ep}\right)\right).
\end{equation*}
We now restrict to  $0<\gamma<2$ and select $k>0$ small enough so that $\gamma<\frac{2}{1+k}$. It follows from the above that, for $0<\ep\ll 1$, 
$$
\left(\frac{1}{2L_\ep},\frac{1+k}{2L_\ep}\right)\subset \left(0,\frac{1}{\gamma ^2\ln \frac 1\ep}\right).
$$
Since
$$
L_\ep^2z+\frac{1}{4z}\leq \frac{L_\ep}{2}(2+k),\;\;\forall z\in \left(\frac{1}{2L_\ep},\frac{1+k}{2L_\ep}\right),
$$
it follows that
\begin{eqnarray*}
H_\ep &\geq & \int_{\frac{1}{2L_\ep}}^{\frac{1+k}{2L_\ep}}\frac{e^{-\frac{L_\ep}2(2+k)}}{z^{\frac{N+2}{2}}}dz  \\
&\geq & \frac{e^{-\frac{L_\ep}2(2+k)}}{\left(\frac{1+k}{2L_\ep}\right)^{\frac{N+2}{2}}}\frac{k}{2L_\ep}=C_{N,k} e^{-\frac{L_\ep}2(2+k)} L_\ep ^{\frac N 2}\sim C_{N,\gamma,k}e^{-\gamma(1+\frac k 2)\ln \frac 1\ep }\left(\ln \frac 1 \ep\right)^{\frac N 2},
\end{eqnarray*}
as $\ep \to 0$. In the above estimate, $C_{N,k}$ and $C_{N,\gamma,k}$ denote positive constants independent of $\ep>0$ small enough, depending on $(N,k)$ and $(N,\gamma,k)$ respectively.
As a result, by restricting further to $0<\gamma<1$ and $k>0$ small enough so that $\gamma(1+\frac k 2)<1$, we get that condition \eqref{condition2} is satisfied for $0<\ep\ll1$. This completes the proof of the result.
\end{proof}

\section{Non extinction in a toy model}\label{s:toy-propagation}

In this section, we fix $\theta\in (0,1)$ and we consider the solution $w=w(t,x)$ of the problem
\begin{equation}\label{EQ3}
w_t=\Delta w+w-\theta,\;t>0,\;x\in\R^N,
\end{equation} 
together with the initial data
\begin{equation}\label{EQ4}
w(0,x)=(\theta+\varepsilon) \mathbf 1_{B_L}(x),\;x\in\R^N.
\end{equation}

\begin{prop}[Non extinction]\label{prop:prop-toy}
Let $\theta<\alpha '<\alpha <1$ be given.  Consider the time 
\begin{equation}
\label{def:Tep2}
T_\varepsilon:=\ln \frac{\alpha-\theta}{\varepsilon}.
\end{equation}
Let $\gamma >2$ be given. Then for any $0<k<1-\frac 2 \gamma$, there exists $\varepsilon_0>0$ small enough such that, for each $\ep\in (0,\ep_0)$ and for each $L>\gamma \ln\frac{1}{\ep}$, the  solution $w=w(t,x)$ of \eqref{EQ3}---\eqref{EQ4} satisfies
\begin{equation*}
\min_{|x|\leq kL_\varepsilon} w(T_\varepsilon,x)\geq \alpha'.
\end{equation*}
\end{prop}

\begin{proof} Notice that the function $w=w(t,x)$ is explicitly given by
\begin{equation}
w(t,x)=\theta+e^{t}(v(t,x)-\theta),
\end{equation}
where $v=v(t,x)$ is the solution of the heat equation
\begin{equation*}
v_t=\Delta v,\;t>0,\;x\in\R^N,
\end{equation*}
with the initial datum $v(0,x)=w(0,x)$. Hence
\begin{eqnarray*}
w(T_\varepsilon,x)&=& \theta+e^{T_\varepsilon}\left[\varepsilon-\frac{\theta+\varepsilon}{(4\pi T_\ep)^{N/2}}\int_{|y|\geq L} e^{-\frac{|x-y|^2}{4T_\ep}}dy\right]\\
&=& \theta+(\alpha-\theta)\left[1-\frac{\theta+\varepsilon}{\varepsilon(4\pi T_\varepsilon)^{N/2}}\int_{|y|\geq L} e^{-\frac{|x-y|^2}{4T_\varepsilon}}dy\right],
\end{eqnarray*}
from the definition of $T_\ep$ in \eqref{def:Tep2}. As result, for all $|x|\leq kL$ one has, since $|y|\geq L$ ensures that $|x|\leq kL\leq k|y|$ and $|x-y|\geq (1-k)|y|$, 
\begin{eqnarray}
w(T_\varepsilon,x)&\geq & \theta+(\alpha-\theta)\left[1-\frac{\theta+\varepsilon}{\varepsilon(4\pi T_\varepsilon)^{N/2}}\int_{|y|\geq L} e^{-\frac{(1-k)^2|y|^2}{4T_\varepsilon}}dy\right]\nonumber\\
&\geq & \theta+(\alpha-\theta)\left[1-\frac{\theta+\varepsilon}{\varepsilon(1-k)^N\pi ^{N/2}}\int_{|z|\geq \frac{(1-k)L}{\sqrt {4T_\varepsilon}} }e^{-|z|^2}dz\right].\label{truc-bidule}
\end{eqnarray}

We now select
\begin{equation}
\label{select-Lep2}
L_\varepsilon=\gamma \sqrt{T_\varepsilon \ln\frac{1}{\varepsilon}},
\end{equation}
for some constant $\gamma>2$. In view of \eqref{truc-bidule}, it is enough to conclude the proof to reach
\begin{equation*}
\theta+(\alpha-\theta)\left[1-\frac{\theta+\varepsilon}{\varepsilon(1-k)^N\pi ^{N/2}}\int_{|z|\geq \frac{(1-k)L_\varepsilon}{\sqrt {4T_\varepsilon}} }e^{-|z|^2}dz\right]\geq \alpha',
\end{equation*}
for $0<\ep \ll 1$,  that is
\begin{equation}\label{machin}
I_\ep:=\int_{|z|\geq \frac{(1-k)}{2}\gamma\sqrt{\ln \frac 1 \ep} }e^{-|z|^2}dz\leq \varepsilon \frac{\alpha-\alpha'}{\alpha-\theta}\frac{(1-k)^N\pi ^{N/2}}{\theta+\varepsilon}.
\end{equation}
On the other hand, by denoting by $A_N>0$ some constant depending on $N$, one has for $X>0$,
\begin{equation*}
\int_{|z|\geq X} e^{-|z|^2}dz=A_N \int_X^\infty r^{N-1}e^{-r^2}dr\sim \frac{A_N}{2} X^{N-2}e^{-X^2}\text{ as $X\to\infty$},  
\end{equation*}
so that, as $\varepsilon\to 0$,
\begin{equation*}
I_\varepsilon \sim \frac{A_N}{2}\left(\frac{(1-k)}{2}\gamma\sqrt{\ln\frac{1}{\varepsilon}}\right)^{N-2} \exp\left(-\frac{(1-k)^2}{4}\gamma^2\ln \frac 1 \ep \right).
\end{equation*}
Since $\frac{(1-k)\gamma}{2}>1$, the above implies $I_\ep=o(\ep)$ as $\ep \to 0$, which validates \eqref{machin}, and thus concludes the proof.
\end{proof}

\section{Quantitative estimates of the (sharp) threshold phenomena}\label{s:sharp}

In this section, we complete the proof of Theorem \ref{th:main} and Corollary \ref{cor:prototype}.

\begin{proof}[Proof of  Theorem \ref{th:main}]
Let Assumption \ref{hyp:nonlinearite} hold. Let us select $\theta<\delta<\min(\delta ^-,\delta ^+)$, where $\delta^-$, $\delta^+$ are provided by Assumption \ref{hyp:nonlinearite} $(i)$, $(ii)$ respectively. 

 By an immediate time-space rescaling argument we have that the threshold values $L_\ep^{ext}=L_\ep^{ext}(f)$ and $L_\ep^{prop}=L_\ep^{prop}(f)$ under investigation are given by
$$
L_\ep^{ext}=\frac 1{\sqrt{r^{+}} }L_\ep^{ext}(f/r^{+}), \quad L_\ep^{prop}=\frac 1{\sqrt{r^{-}}}L_\ep^{prop}(f/r^{-}),
$$
where $L_\ep^{ext/prop}(f/r)$ are the threshold values associated with the nonlinearity $f/r$. 

Let us first enquire on $L_\ep^{ext}(f/r^{+})$. Let us denote by $w_L^{\ep}=w_L^{\ep}(t,x)$ the solution to
$$
w_t=\Delta w +\frac{1}{r^{+}}f(w),
$$
starting from $\phi_L^{\ep}(x)=(\theta+\ep)\mathbf 1_{B_L}(x)$. By combining Assumption \ref{hyp:nonlinearite} $(i)$  and comparison with the ordinary differential equation (use the supersolution $t\mapsto \theta+\ep e^t$), we immediately have the upper bound
$$
w(t,x)\leq\delta,\; \forall (t,x) \in [0,T_\ep]\times \R^N,
$$
where the time $T_\ep$ was defined in \eqref{def:Tep}. This in turn implies that $w$ is a subsolution to \eqref{EQ} on $(0,T_\ep)\times \R^{N}$. As a result, the extinction result Proposition \ref{prop:extinction-toy} applies: for any $0<\gamma<1$, we have, for $0<\ep\ll 1$, $L_\ep^{ext}(f/r^{+})\geq \gamma \ln \frac 1 \ep$ and thus
$$
L_\ep^{ext}=L_\ep^{ext}(f)\geq \frac{\gamma}{\sqrt{r^{+}}}\ln \frac 1 \ep.
$$
This rewrites as
\begin{equation*}
\frac{1}{\sqrt {r^+}}\leq \liminf_{\ep\to 0}\frac{L_\ep^{ext}}{\ln \frac{1}{\ep}}.
\end{equation*}

Let us now enquire on $L_\ep^{prop}(f/r^{-})$. Let us denote by $w_L^{\ep}=w_L^{\ep}(t,x)$ the solution to
$$
w_t=\Delta w +\frac{1}{r^{-}}f(w),
$$
starting from $\phi_L^{\ep}(x)=(\theta+\ep)\mathbf 1_{B_L}(x)$. 
Set $g(w)=\frac{1}{r^{-}}f(w)$ and note that due to Assumption \ref{hyp:nonlinearite} $(ii)$, one has
$$
g(w)\geq w-\theta,\;\forall w\in (-\infty,\delta].
$$
Let $\eta\in (0,\delta-\theta)$ be small enough and define a Locally Lipschitz continuous function $\widehat{g}:\R\to \R$ such that
\begin{equation*}
\widehat{g}(w)=\begin{cases} g(w) &\text{ for }w\in ]-\infty,\theta]\cap [\delta,\infty),\\
w-\theta &\text{ for }w\in [\theta,\theta+\eta],
\end{cases}
\end{equation*}
$\widehat g(w)\leq g(w)$ for $w\in [\theta+\eta,\delta]$ and $\int_0^1 \widehat g(s)ds>0$. Denote $\widehat w=\widehat w(t,x)$ the solution of the problem
$$
\widehat w_t=\Delta \widehat w +\widehat{g}(\widehat{w}),\;\;\widehat w(0,x)=w(0,x),
$$
so that $\widehat w(t,x)\leq w(t,x)$.
Fix $\eta'\in (0,\eta)$ and consider the time $T_\ep=\ln \frac{\eta}{\ep}$. Since 
$$
\widehat g(w)\geq w-\theta,\;\forall w\in (-\infty,\theta+\eta],
$$
and for all $\ep\leq \eta$ one has $\widehat w(t,x)\leq \theta+\eta$ on $[0,T_\ep]\times\R^N$, 
using Proposition \ref{prop:prop-toy}, one obtains that, for any given $\gamma>2$ and $k>1-\frac{2}{\gamma}$, there exists $\ep_0>0$ such that, for all $\ep\in (0,\ep_0)$ and $L>L_\ep:=\gamma\ln \frac{1}{\ep}$, one has
$$
\widehat w(T_\ep,x)\geq (\theta+\eta'){\bf 1}_{B_{kL_\ep}}.
$$
Define by $\widehat L^{prop}_{\eta'}$ the propagating threshold associated to the equation with the nonlinearity $\widehat g$.
Hence, for all $\ep$ small enough such that $kL_\ep>\widehat L^{prop}_{\eta'}$, one has
$\widehat w(t,x)\to 1$ as $t\to+\infty$ locally uniformly in space and therefore $w(t,x)\to 1$ as $t\to+\infty$ locally uniformly in space. As a consequence we obtain
\begin{equation*}
L^{prop}_\ep(f/r^-) \leq L_\ep\text{ for all $\ep<\!\!<1$}.
\end{equation*}
Hence we get
\begin{equation*}
\limsup_{\ep\to 0}\frac{L_\ep^{prop}(f/r^-)}{\ln \frac{1}{\ep}}\leq \gamma,\;\forall \gamma>2.
\end{equation*}
This rewrites as
\begin{equation*}
\limsup_{\ep\to 0}\frac{L_\ep^{prop}}{\ln \frac{1}{\ep}}\leq \frac{2}{\sqrt {r^-}},
\end{equation*}
and concludes the proof of Theorem \ref{th:main}.
\end{proof}

\begin{proof}[Proof of Corollary \ref{cor:prototype}] In view of Corollary \ref{cor:sharp}, it is sufficient to prove that the threshold is sharp. In the bistable case, the threshold is known to be sharp when we further assume that $f$ is of the class $C^1$ and $f'(0)<0$, see \cite{Pol-11}, which is the case of the usual cubic nonlinearity \eqref{bistable}. On the other hand, the ignition case \eqref{ignition} can be handled using the results of \cite{Mur-Zho-17}. Indeed fix $\ep\in (0,1-\theta)$ and fix $\overline{L}>L_\ep^{prop}$. We consider the family of solutions $u^\ep_L=u_L^\ep(t,x)$ of \eqref{eq} starting from \eqref{initial-data}. Recall that it is non decreasing with respect to $L$.
Then, since $u^\ep_{\overline L}(t,\cdot)\to 1$ as $t\to+\infty$, locally uniformly in $\R^N$, there exists $\overline{t}>0$ such that the energy, see \eqref{energy}, satisfies
$$
\mathcal E\left(u^\ep_{\overline L}(\overline{t},\cdot)\right)<0.
$$
Finally, since $u^\ep_L(\overline{t},\cdot)\to 0$ as $L\to 0$, the existence of the sharp threshold $L_\ep^\ep$ directly follows from the application of the results of Muratov and Zhong \cite{Mur-Zho-17} for ignition nonlinearities and the family of initial data $\{u^\ep_L(\overline t,\cdot)\}_{L>0}$. Notice that this argument also applies in the bistable case and in some monostable situations (see Corollary \ref{cor:sharp-mono}). 
\end{proof}

\section{Numerical explorations}\label{s:numeric}

We explored sharp threshold of propagation numerically for the one dimensional case, that is $N=1$ throughout this section. We first report on numerical explorations of the validity range of our predictions from Corollary \ref{cor:sharp}  for nonlinearities satisfying Assumption~\ref{hyp:nonlinearite} and more precisely Corollary \ref{cor:prototype} for \eqref{ignition} and \eqref{bistable}. We also show some evidence for the existence of a universal constant $C(\theta)$ such that 
\begin{equation*}
\underset{\ep\rightarrow0^+}{\lim}~\frac{L_\ep^\star}{\ln\frac{1}{\ep}}=C(\theta),
\end{equation*}
for both ignition and bistable nonlinearities. Finally, we investigate other dependencies.

\subsection{Numerical validation of Corollary \ref{cor:prototype}}

We used second order finite differences in space and a splitting method in time where the diffusion part of the equation is evaluated with a Crank-Nicolson scheme while the nonlinear term is either evaluated exactly (for the ignition case) or with a Runge-Kuttta method of order two (for the bistable nonlinearity). Time step and space discretization were set respectively to $\delta t = 0.02$ and $\delta x = 0.02$. We used direct numerical simulations to evaluate the critical threshold $L_\ep^\star$ using the following strategy. For fixed nonlinearity $f$ and $\ep>0$, we varied the size $L>0$ of the support of the initial condition:
\begin{equation*}
\phi_L^\ep=(\theta+\ep)\mathbf{1}_{[-L,L]},
\end{equation*}
and ran our numerical scheme from $t=0$ to some final time $t=T$. For each $L>0$, we evaluated the maximum of the numerically computed solution at the last time step. We then discriminated the precise value of $L$ for which there is a transition from extinction (maximum being close to zero) to propagation (maximum being close to one). As we expect the sharp threshold $L_\ep^\star$ to be of the order $\ln\frac{1}{\ep}$ as $\ep\rightarrow 0^+$ and the transition to occur for times of the order $\ln\frac{1}{\ep}$, we took a spatial domain of computation $[-100,100]$ and a final time of computation $T=100$. This allowed us to cover a range for $\ep$ from $10^{-3}$ to $10^{-1}$.

\paragraph{Ignition nonlinearity.} We present our results for the ignition nonlinearity function
\begin{equation*}
f(u)=(u-\theta)(1-u), \quad u\in[0,1],
\end{equation*}
for several values of $\theta\in(0,1)$. Numerically computed values of $L_\ep^\star$ are plotted in Figure~\ref{fig:lnEps} (left) as a function of $\ln\frac{1}{\ep}$ together with a corresponding linear fit for several representative values of $\theta$. Measured slopes are presented in Table~\ref{tableIgn} for a larger range of $\theta$ ranging from $0.1$ to $0.6$. Our numerical results seem to indicate that 
\begin{equation*}
\underset{\ep\rightarrow0^+}{\lim}~\frac{L_\ep^\star}{\ln\frac{1}{\ep}}=\frac{1}{\sqrt{1-\theta}},
\end{equation*}
as it can be inferred from the last line in Table~\ref{tableIgn}.

\begin{table}[!h] 
\begin{center}
\begin{tabular}{|c||c|c|c|c|c|c|c|}
\hline  & $\theta=0.1$  & $\theta=0.2$ & $\theta=0.3$ & $\theta = 0.4 $ & $\theta = 0.5$ & $\theta=0.6$  \\ \hline
$s_m$ & 1.11&  $1.20$  & $1.24$ & $1.32$ & 1.43 & 1.57  \\\hline
$  \sqrt{1-\theta}s_m$ & 1.05 &  $1.07$ &  1.04   & $1.02$ & 1.01 & 0.99  \\\hline
\end{tabular}
\end{center}
\caption{{\bf Ignition nonlinearity.} Measured slopes $s_{m}$ of the linear fit of $L_\ep^\star$ as a function of $\ln\frac{1}{\ep}$ together with their rescaled values $\sqrt{1-\theta}s_m$. See Figure~\ref{fig:lnEps} (left). Note that $\sqrt{1-\theta}s_m\simeq 1$ appears to be independent of $\theta$.}
\label{tableIgn} 
\end{table}

 \begin{figure}[t!]
\centering 
\includegraphics[width=0.49\textwidth]{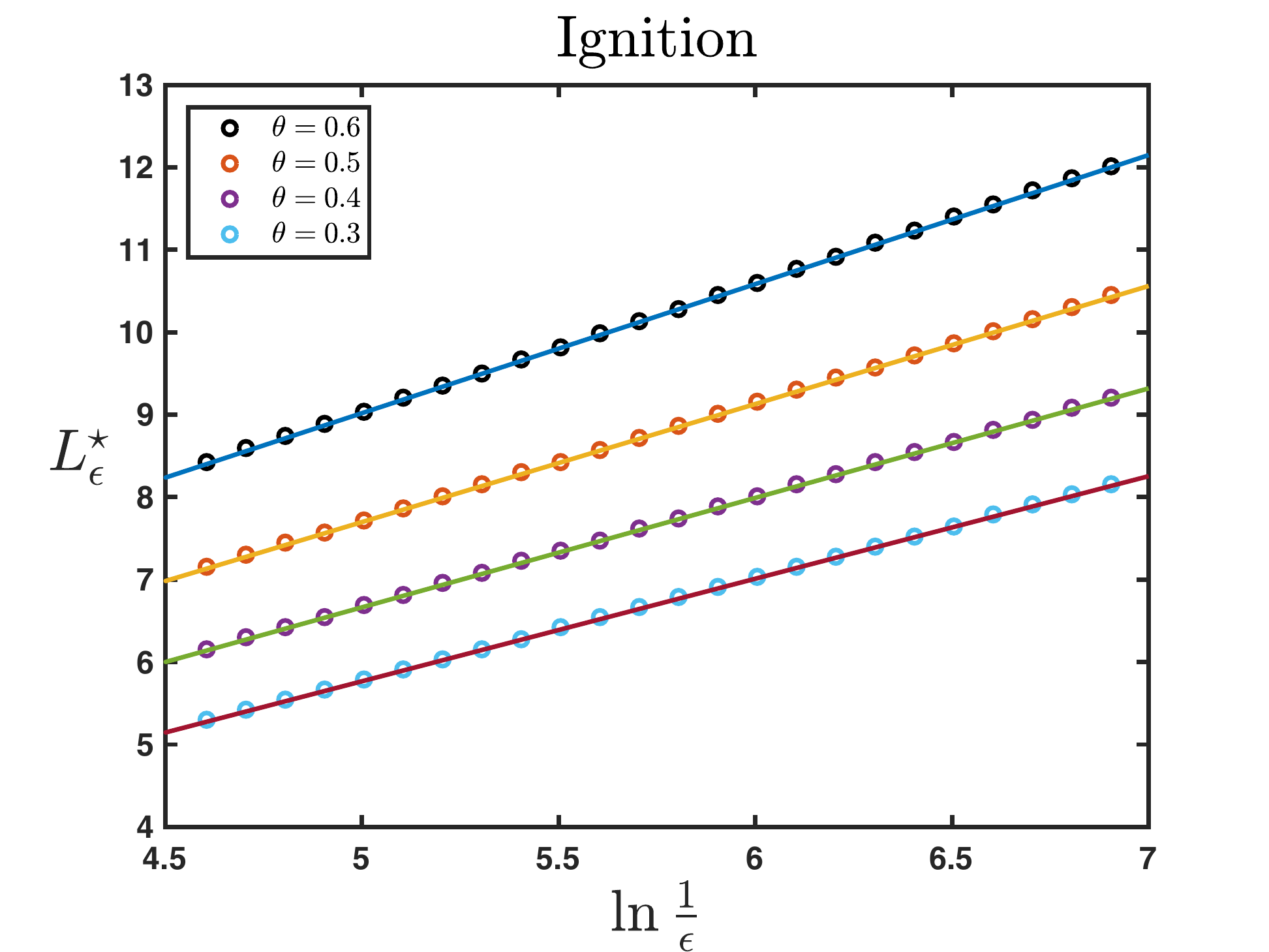}  
 \includegraphics[width=0.49\textwidth]{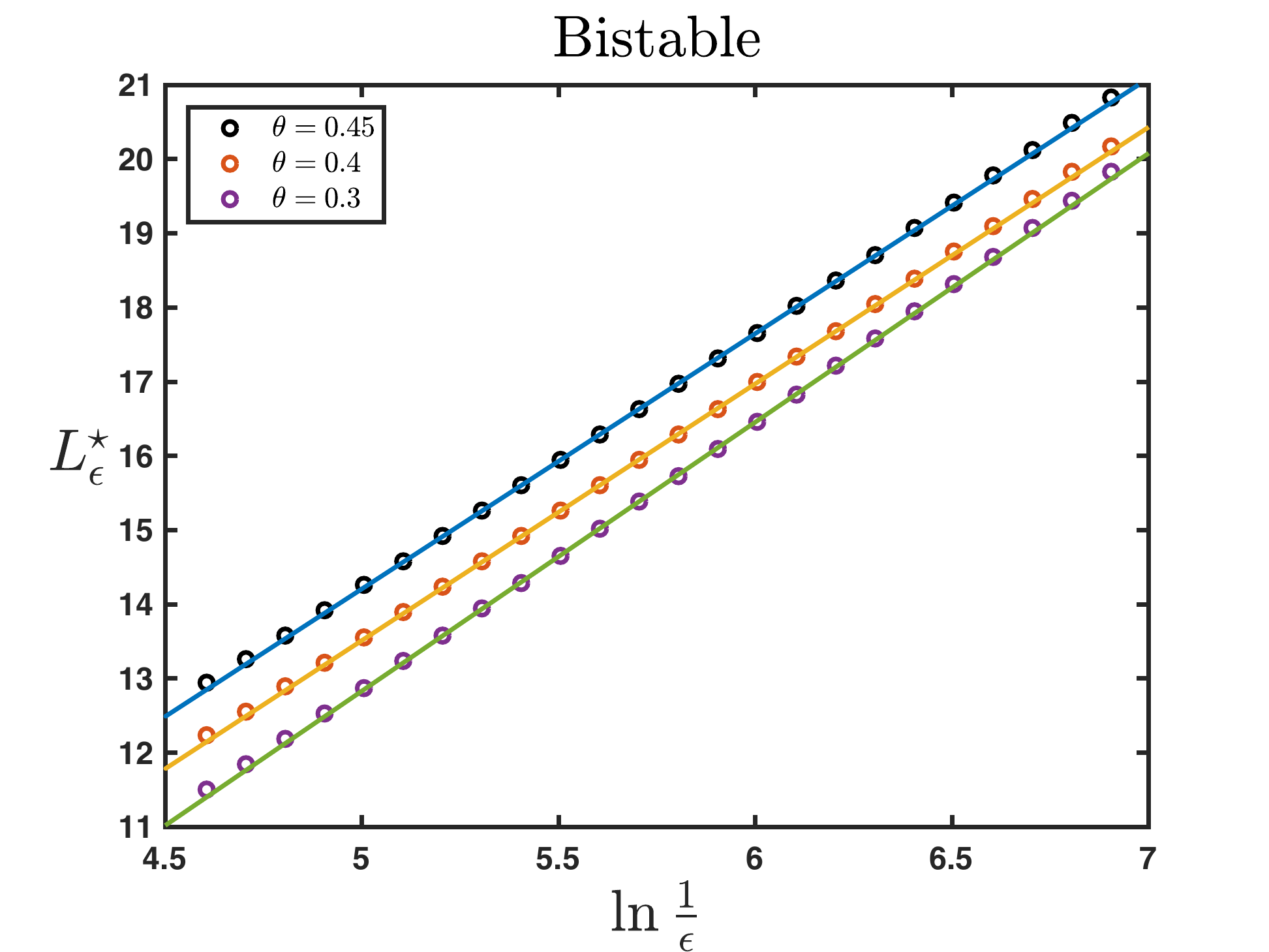}
\caption{Numerically computed values of $L_\ep^\star$ as a function of $\ln\frac{1}{\ep}$ for several values of $\theta$ for the ignition nonlinearity (left) and bistable nonlinearity (right). We also superimposed in plane line a linear fit whose slopes $s_m$ are reported in Table~\ref{tableIgn} for the ignition nonlinearity (resp. in Table~\ref{tableBis} for the bistable nonlinearity).}
\label{fig:lnEps}
\end{figure}

\paragraph{Bistable nonlinearity.} We present our results for the bistable nonlinearity function
\begin{equation*}
f(u)=u(u-\theta)(1-u), \quad u\in[0,1],
\end{equation*}
for several values of $\theta\in(0,1/2)$. Numerically computed values of $L_\ep^\star$ are plotted in Figure~\ref{fig:lnEps} (right) as a function of $\ln\frac{1}{\ep}$ together with a corresponding linear fit for several representative values of $\theta$. Measured slopes are presented in Table~\ref{tableBis} for a larger range of $\theta$ ranging from $0.1$ to $0.4$. As predicted by our main Theorem~\ref{th:main}, we observe that $\sqrt{\theta(1-\theta)}s_m\in [1,2]$. Contrary to the ignition case, it seems that $\sqrt{\theta(1-\theta)}s_m$ does depend on $\theta$ and, moreover, in a nonlinear fashion. Our measured values did not allow us to conjecture an elaborated guess for this dependence. We report in Figure~\ref{fig:Ctheta} the numerically computed values of $\sqrt{\theta(1-\theta)}s_m$ for several values of $\theta\in(0,1/2)$.

\begin{table}[!h] 
\begin{center}
\begin{tabular}{|c||c|c|c|c|c|c|c|}
\hline & $\theta = 0.1 $ & $\theta = 0.15$ & $\theta = 0.2 $ & $\theta = 0.25$ & $\theta = 0.3 $ & $\theta = 0.35 $ & $\theta = 0.4$ \\ \hline
$s_m$  & $5.12$ & $4.43$ & $4.04$ & $3.78$ & $3.62$ & $3.52$ &  $3.45$   \\\hline
$  \sqrt{\theta(1-\theta)}s_m$  & $1.53$ & $1.58$ & $1.61$ & $1.63$  & $1.66$  & $1.67$ & $1.69$  \\\hline
\end{tabular}
\end{center}
\caption{{\bf Bistable nonlinearity.} Measured slopes $s_{m}$ of the linear fit of $L_\ep^\star$ as a function of $\ln\frac{1}{\ep}$  together with their rescaled values $\sqrt{\theta(1-\theta)}s_m$.  See Figures \ref{fig:lnEps} (right) and \ref{fig:Ctheta}. Note that $\sqrt{\theta(1-\theta)}s_m\in [1,2]$ as predicted by our main Theorem~\ref{th:main}.}
\label{tableBis} 
\end{table}

\subsection{Other dependencies -- Bistable nonlinearity}
\begin{figure}[!t]
\centering 
\includegraphics[width=0.5\textwidth]{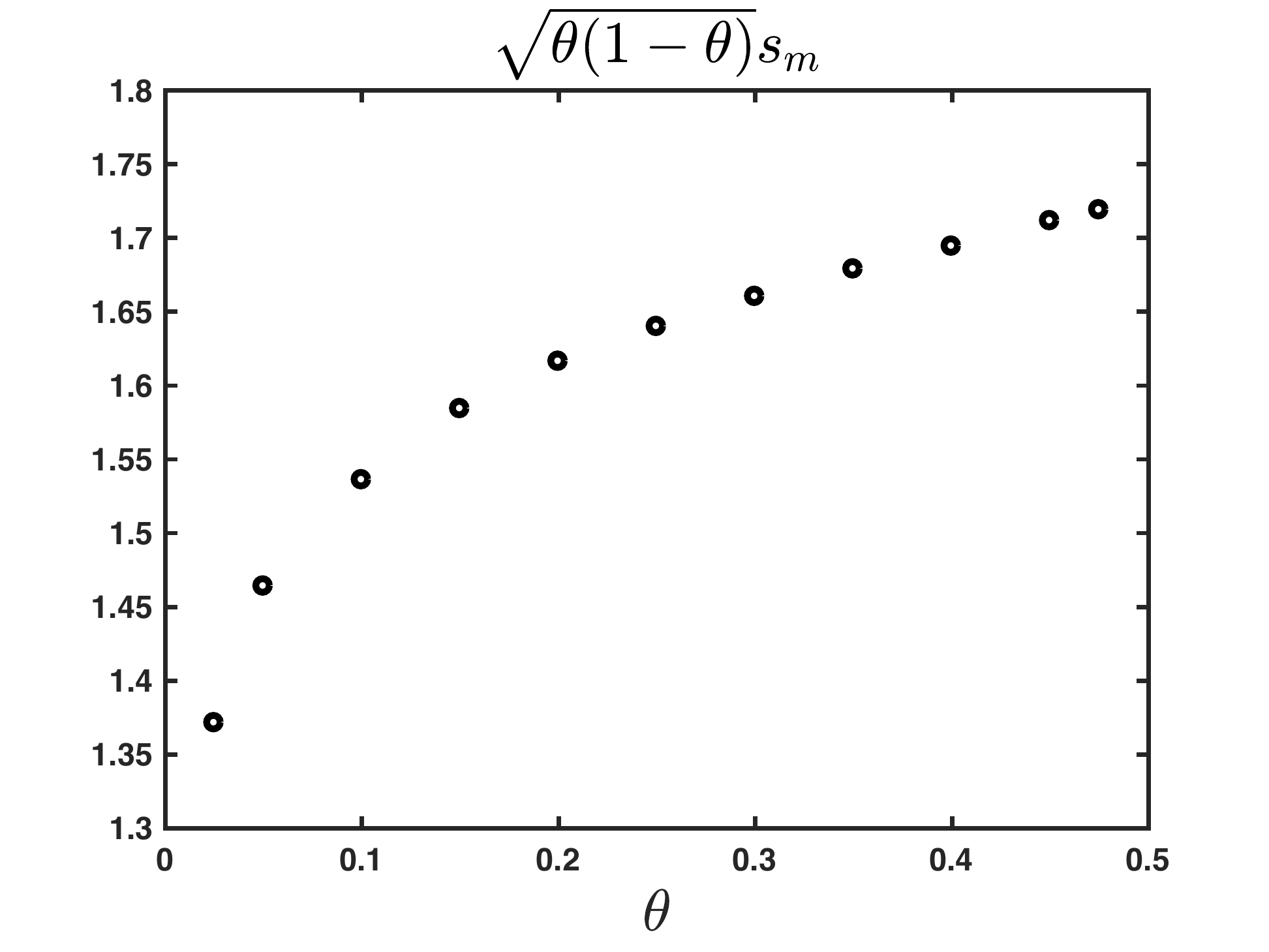}  
\caption{Numerically computed values of $\sqrt{\theta(1-\theta)}s_m$ for several values of $\theta\in(0,1/2)$ as reported in Table~\ref{tableBis} for the bistable nonlinearity \eqref{bistable}.}
\label{fig:Ctheta}
\end{figure}

Here, we report on some other dependencies of the critical threshold $L^\star$ in the case of the bistable nonlinearity 
\begin{equation*}
f(u)=u(u-\theta)(1-u).
\end{equation*}
We have considered the following family of compactly supported initial conditions:
\begin{equation*}
\phi_L=\mathbf{1}_{[-L,L]}, \quad  L>0,
\end{equation*}
such that the amplitude of the initial condition is fixed to the stable steady state $1$. For each $\theta\in(0,1/2)$, we numerically computed the corresponding critical threshold $L^\star$ by varying the size $L>0$ of the support. We report in Figure~\ref{fig:BisThetaVar} our numerical findings. We recover that in the limit $\theta\rightarrow0^+$, the critical threshold $L^\star \rightarrow 0$. Indeed, as  $\theta\rightarrow0^+$, the nonlinearity approaches $f(u)=u^2(1-u)$ for which the hair trigger effect is known \cite{Aro-Wei-78}. Our numerics suggest that the dependence of $L^\star$ near $\theta=0$ is linear, and we get
\begin{equation*}
\underset{\theta\rightarrow0^+}{\lim}~\frac{L^\star}{\theta} \approx 4.85.
\end{equation*}
On the other hand, when $\theta\rightarrow \frac{1}{2}^-$, we obtain $L^\star\rightarrow+\infty$. In the limit, we have $\int_0^1f(u)du=0$, and there does not exist any ground state (homoclinic solution to the steady state $0$) to equation \eqref{eq}. Instead, there is a one-parameter family of stationary monotone interfaces between $u=0$ and $u=1$. Our numerics suggest that the dependence of $L^\star$ near $\theta=\frac{1}{2}$ is logarithmic, and we get
\begin{equation*}
\underset{\theta\rightarrow\frac{1}{2}^-}{\lim}~\frac{L^\star}{\ln \frac{1}{1-2\theta}} \approx 0.99.
\end{equation*}

\begin{figure}[t!]
\centering 
\includegraphics[width=0.5\textwidth]{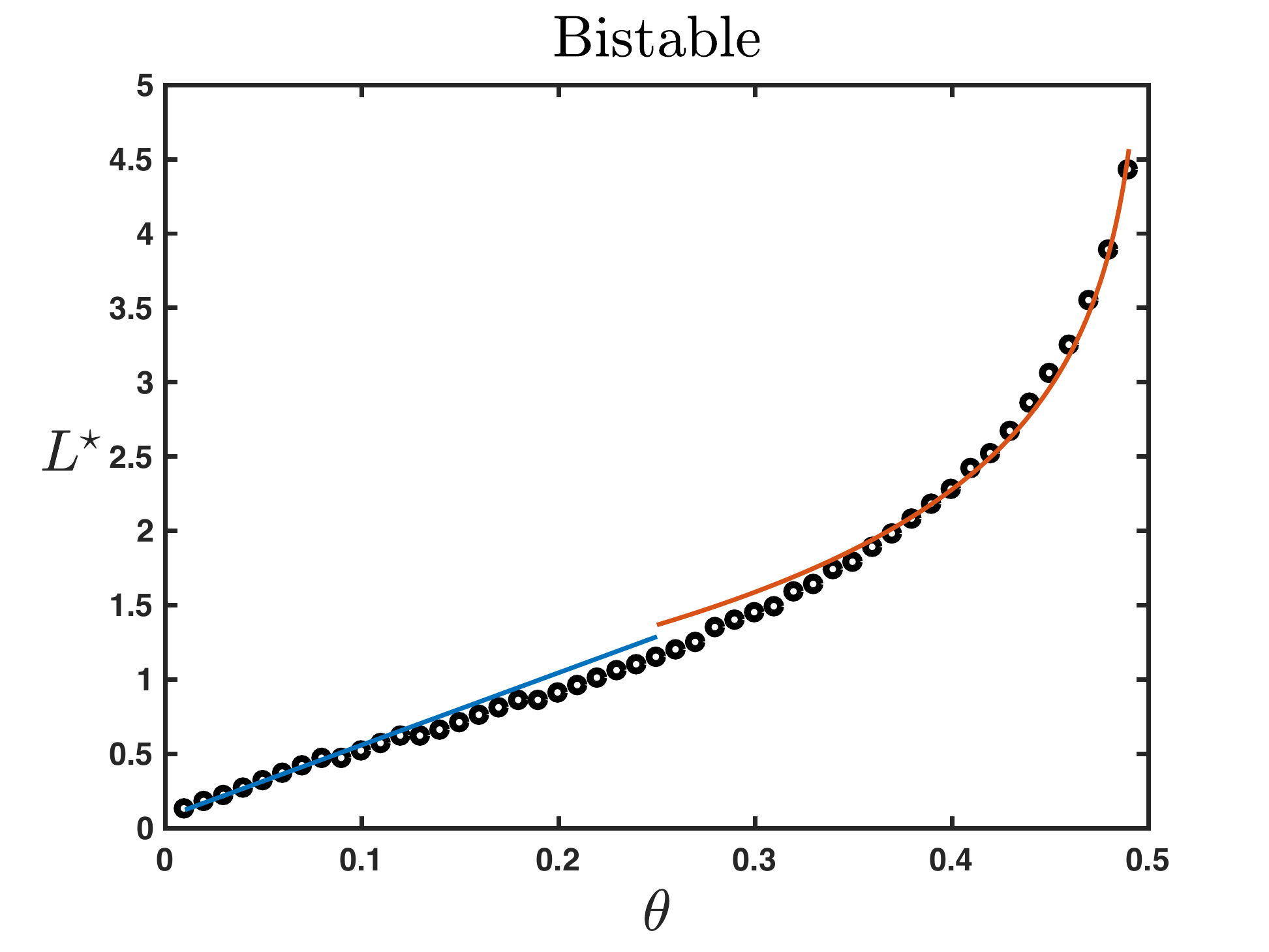}  
\caption{Numerically computed values of $L^\star$ as a function of $\theta$ for the bistable nonlinearity and amplitude of the initial condition fixed to the stable steady state $1$. We also superimposed in plane lines a linear fit near $\theta=0$ (blue line) and a logarithmic fit near $\theta=1/2$ (red line). The measured slope $s_m^0$ of the linear fit near $\theta=0$ is $s_m^0=4.85$. The measured slope $s_m^{1/2}$ of logarithmic fit of $L^\star$ as a function of $\ln \frac{1}{1-2\theta}$ is $s_m^{1/2}=0.99$.}
\label{fig:BisThetaVar}
\end{figure}

\section{Degenerate monostable case}\label{s:degenerate}

In this section we consider $u=u(t,x)$ the solution of the Cauchy problem
\begin{equation}\label{eq:deg-mono}
\begin{cases}
u_t=\Delta u+f(u), &\;t>0,\;x\in\R^{N},\\
u(0,x)=\ep{\mathbf 1}_{B_L}(x), &\;x\in \R^{N},
\end{cases}
\end{equation}
where $\ep\in (0,1]$ and $L>0$ are given parameter, and where $f$ is a monostable degenerate nonlinearity, that is assumed to satisfy the following set of assumptions.

\begin{assumption}[Degenerate monostable nonlinearity $f$]\label{hyp:deg-mono}
The function $f:\R\to\R$ is locally Lipschitz continuous, satisfies $f=0$ on $(-\infty,0]\cup [1,+\infty)$,
\begin{equation*}
f(u)>0,\;\forall u\in (0,1),
\end{equation*}
and
\begin{equation}
0<\liminf_{u\to 0^+}\frac{f(u)}{u^p}\leq \limsup_{u\to 0^+}\frac{f(u)}{u^p}<+\infty,
\end{equation}
for some $p>1$.
\end{assumption}

When $1<p\leq p_F$, where $p_F:=1+\frac 2 N$ is the so called Fujita exponent \cite{Fuj-66}, it is well known \cite{Aro-Wei-78} that the reaction-diffusion equation $u_t=\Delta u+f(u)$ enjoys the hair trigger effect: for any nonnegative and nontrivial initial data, the solution tends to 1 as $t\to +\infty$, locally uniformly in $x\in \R^{N}$. On the other hand, when $p> p_F$, it is known \cite{Aro-Wei-78} that some \lq\lq small enough'' initial data imply extinction, whereas some \lq\lq large enough'' initial data imply propagation.  Our goal is here to provide some quantitative estimates on this threshold phenomenon.

The first result of this section provides an lower estimate for the radius of extinction.

\begin{theo}[Extinction]\label{th:extinction-deg-mono} Let Assumption \ref{hyp:deg-mono} hold with $p>1+\frac 2 N$. Then there is a  constant $C^{-}=C^{-}(f,N)>0$ such that for all $\ep\in (0,1]$, as soon as
\begin{equation}
\label{condition-ext}
L<\frac{C^{-}}{\ep^{\frac{p-1}{2}}},
\end{equation}
the solution to \eqref{eq:deg-mono} satisfies
\begin{equation}
\label{extinction}
\Vert u(t,\cdot)\Vert _{L^{\infty}(\R^{N})} =\mathcal{O}\left(t^{-\frac N 2}\right),  \quad \text{ as } t\to +\infty.
\end{equation}
\end{theo}

\begin{proof} 
Throughout this proof we fix $\ep\in (0,1]$ and 
we consider the solution $v=v(t,x)$ of the heat equation
\begin{equation*}
v_t=\Delta v,\quad t>0,\;x\in\R^{N},  
\end{equation*}
starting from $v(0,\cdot)=u(0,\cdot)=\ep \mathbf 1 _{B_L}$, so that
\begin{equation*}
V(t):=\Vert  v(t,\cdot)\Vert _{L^\infty(\R^N)}=\frac{\ep}{(4\pi)^{\frac N 2}}\int_{\vert x\vert <\frac{L}{\sqrt t}} e^{-\frac{\vert x\vert ^2}{4}}dx=:\ep A_L(t).
\end{equation*}
Following the proof of Proposition \ref{prop:extinction-toy}, we seek a supersolution to \eqref{eq} for $t>0$ and $x\in\R^N$ in the form $U(t,x):=v(t,x)\varphi(t)$, with $\varphi(0)=1$ and $\varphi(t)>0$. Since there is $K>0$ such that $f(u)\leq K u^{p}$ for all $u\in \R$, it is enough to have
$$
\frac{\varphi'(t)}{\varphi^{p}(t)}=K \ep^{p-1}A_L^{p-1}(t),
$$
which is solved as
$$
\varphi(t)=\frac{1}{\left(1-(p-1)K \ep ^{p-1}\int _0 ^t A_L^{p-1}(s)ds\right)^{\frac{1}{p-1}}}.
$$
Since we need $\varphi$ to be global, the condition we obtain reads as
$$
1>(p-1)K \ep^{p-1}\int _0 ^{+\infty} A_L^{p-1}(s)ds.
$$
Notice that, for a given $L>0$, there exists some positive constant $C_{N,L}$ such that $A_L(t)\sim \frac{C_{N,L}}{t^{\frac N 2}}$ as $t\to +\infty$ so that the above improper integral does converge since  $p>1+\frac 2 N$. 
The above condition rewrites as
\begin{equation*}
1>(p-1)K \ep^{p-1} \frac 1{(4\pi)^{\frac{(p-1)N}{2}}}\int _0 ^{+\infty} \left(\int_{\vert x\vert <\frac{L}{\sqrt s}} e^{-\frac{\vert x\vert ^2}{4}}dx\right)^{p-1}ds,
\end{equation*}
or, letting $s=L^2 t$,
\begin{equation*}
1>(p-1)K \ep^{p-1} \frac 1{(4\pi)^{\frac{(p-1)N}{2}}} L^{2}\int _0 ^{+\infty} \left(\int_{\vert x\vert <\frac{1}{\sqrt t}} e^{-\frac{\vert x\vert ^2}{4}}dx\right)^{p-1}dt=:CL^{2}\ep^{p-1},
\end{equation*}
for some constant $C=C(p,K,N)>0$. Hence, when this condition is fulfilled, we are equipped with a global supersolution $U(t,x)=v(t,x)\varphi(t)$ with $\varphi$ bounded which ensures that
\eqref{extinction} holds. This concludes the proof.
\end{proof}

Our next result is concerned with a lower estimate for the radius leading to propagation. Contrary to the previous result that provides an estimate valid for any $\ep\in (0,1]$, here our lower estimate is valid for $0<\ep\ll 1$ and it reads as follows.

\begin{theo}[Propagation]\label{th:prop-deg-mono}  Let Assumption \ref{hyp:deg-mono} hold. 
Then there exist  $C^{+}=C^{+}(f,N)>0$ and  $\ep_0>0$ small enough such that, for all $\ep\in (0,\ep_0)$, there is $L_\ep>0$ satisfying
\begin{equation}
\label{condition-prop}
L_\ep\sim  \frac {C^{+}} {\ep^{\frac{p-1}2}} \left(\ln\frac{1}{\ep}\right)^{\frac 12}\; \text{ as } \ep\to 0,
\end{equation}
such that, for all $L>L_\ep$, the solution to \eqref{eq:deg-mono} satisfies
\begin{equation}
\label{qqch}
\lim _{t\to +\infty} u(t,x)= 1\text{ locally uniformly in } \R^{N}.
\end{equation}
\end{theo}

\begin{proof}
From Assumption \ref{hyp:deg-mono}, there are $\delta>0$ and $r>0$ such that $f(u)\geq r u^p$ for all $0\leq u\leq \delta$. Up to a time-space rescaling argument we can assume, without loss of generality, that $r=1$.
As a direct consequence of Proposition \ref{prop:threshold} (for the ignition case), Problem \eqref{eq:deg-mono} enjoys the threshold phenomena that ensures that for any $\alpha\in (0,1)$ there exists a radius $\widetilde L_\alpha>0$ such that for all $L\geq \widetilde L_\alpha$ the function $w_L=w_L(t,x)$ defined as the resolution of \eqref{eq:deg-mono} with the initial datum $w(0,x)=\alpha {\bf 1}_{B_L}(x)$ satisfies 
$$
w_L(t,x)\to 1\text{ as $t\to+\infty$ locally uniformly for $x\in\R^N$}.
$$

Now to prove Theorem \ref{th:prop-deg-mono}, let us consider $Y(\tau,\xi),$ the solution of the ODE Cauchy problem, where $\xi \in \R$ serves as a parameter,
$$
\partial_\tau Y(\tau,\xi)=(Y)_+^p,\quad \;Y(0,\xi)=\xi,
$$
namely
$$
Y(\tau,\xi)=\begin{cases} \xi &\text{ if $\tau>0$ and } \xi\leq 0,\\
\frac{1}{\left(-(p-1)\tau+\frac{1}{\xi^{p-1}}\right)^{\frac{1}{p-1}}} &\text{ if $\tau>0$, $\xi>0$ and $\tau<\frac 1{(p-1)\xi^{p-1}}$}.
\end{cases}
$$
For $0<\ep<\delta$ we set
\begin{equation}\label{def-T-ter}
T_\ep:=\frac{1}{(p-1)\ep^{p-1}}-\frac{1}{(p-1)\delta^{p-1}}.
\end{equation}
And observe that one has
$$
0<Y(\tau,\xi)<\delta, \quad \forall \tau\in(0,T^\ep), \forall \xi\in(0,\ep).
$$
Again we consider the solution $v=v(t,x)$ of the heat equation starting from $v(0,\cdot)=u(\cdot,0)=\ep\mathbf 1_{B_L}$, and define
$$
u^-(t,x):=Y(t,v(t,x)), \quad 0\leq t\leq T^\ep, x\in \R^{N}.
$$
For $0<t\leq T^\ep$, $x\in \R^N$, we have
\begin{eqnarray*}
u^-_t-\Delta u^--f(u^-)&=& Y_\tau+v_t Y_\xi -\Delta v Y_\xi-\vert \nabla v\vert ^2 Y_{\xi\xi}-f(Y)\\
&\leq &Y_\tau+v_t Y_\xi -\Delta v Y_\xi-\vert \nabla v\vert ^2 Y_{\xi\xi}-Y^{p}\\
&=& -\vert \nabla v\vert ^2 Y_{\xi\xi}\\
&\leq & 0,
\end{eqnarray*}
since one can easily check that $Y_{\xi\xi}(\tau,\xi)\geq 0$ for all $\tau>0$, $\xi>0$ with $\tau<\frac 1{(p-1)\xi^{p-1}}$. Hence, for each $\ep\in (0,\delta)$, the comparison principle applies and yields
\begin{equation}
\label{u-plus-grand}
u(T_\ep,x)\geq Y(T_\ep,v(T_\ep,x)), \quad \forall x\in \R^N.
\end{equation}

Next, we have 
$$
v(T_\varepsilon,x)=\varepsilon-\frac{\varepsilon}{(4\pi T_\ep)^{N/2}}\int_{|y|\geq L} e^{-\frac{|x-y|^2}{4T_\ep}}dy.
$$
As a result, for any $k\in(0,1)$, for all $|x|\leq kL$ one has, since $|y|\geq L$ ensures that $|x|\leq kL\leq k|y|$ and $|x-y|\geq (1-k)|y|$, 
\begin{eqnarray}
v(T_\varepsilon,x)&\geq & \ep-\frac{\varepsilon}{(4\pi T_\varepsilon)^{N/2}}\int_{|y|\geq L} e^{-\frac{(1-k)^2|y|^2}{4T_\varepsilon}}dy\nonumber\\
&\geq & \ep-\frac{\varepsilon}{(1-k)^N\pi ^{N/2}}\int_{|z|\geq \frac{(1-k)L}{\sqrt {4T_\varepsilon}} }e^{-|z|^2}dz:=\xi_\ep.\label{truc}
\end{eqnarray}

Now fix $0<\delta'<\delta$ and $k\in (0,1)$. We look for a condition ensuring that, for all $x\in B_{kL}$, we have $Y(T_\ep,v(T_\ep,x))\geq \delta'$, that reads as 
$$
\left[-(p-1) T_\ep+\xi_\ep^{1-p}\right]^{\frac{1}{1-p}}\geq \delta'.
$$
Recalling the definition of $T_\ep$ in \eqref{def-T-ter} this rewrites as
$$
\xi_\ep\geq \ep \left[1+\ep^{p-1}\left[\frac{1}{\left(\delta'\right)^{p-1}}-\frac{1}{(\delta)^{p-1}}\right]\right]^{\frac{1}{1-p}}.
$$
Denoting $C:=\left[\frac{1}{\left(\delta'\right)^{p-1}}-\frac{1}{(\delta)^{p-1}}\right]>0$, the above inequality becomes
$$
\xi_\ep\geq \ep \left[1+\ep^{p-1}C\right]^{\frac{1}{1-p}}=\ep-\frac{\ep^{p}C}{p-1}+\mathcal O\left(\ep^{2p-1}\right).
$$
In view of \eqref{truc} it is sufficient to have, for $0<\ep\ll1$,
\begin{equation}\label{truc-bidule-bis}
\frac{\varepsilon}{(1-k)^N\pi ^{N/2}}\int_{|z|\geq \frac{(1-k)L}{\sqrt {4T_\varepsilon}} }e^{-|z|^2}dz\leq \frac{\ep^{p}C}{p-1}+\mathcal O\left(\ep^{2p-1}\right).
\end{equation}
We now select
\begin{equation}
\label{select-Lep2-bis}
L_\varepsilon=\gamma \sqrt{T_\varepsilon \ln\frac{1}{\varepsilon}},
\end{equation}
for some constant $\gamma>0$. 
As in the proof of Proposition \ref{prop:prop-toy}, if $I_\ep$ denotes the integral in the left hand side of \eqref{truc-bidule-bis}, we have, as $\varepsilon\to 0$,
\begin{equation*}
I_\varepsilon \sim \frac{A_N}{2}\left(\frac{(1-k)}{2}\gamma\sqrt{\ln\frac{1}{\varepsilon}}\right)^{N-1} \exp\left(-\frac{(1-k)^2}{4}\gamma^2\ln \frac 1 \ep \right),
\end{equation*}
for some constant $A_N>0$ independent of $\ep$ small enough.
Using this we see that \eqref{truc-bidule-bis} holds true, for $0<\ep\ll1$, as soon as $\gamma>0$ is large enough so that
\begin{equation}\label{condition-ter}
\frac{(1-k)^2}{4}\gamma^2>p-1.
\end{equation}

We have thus proved that, for any $k\in(0,1)$ and $\gamma >0$ such that \eqref{condition-ter} holds, we have that, for all $0<\ep\ll 1$, 
$$
\min _{\vert x\vert \leq kL_\ep}u(T_\ep,x)\geq \alpha '.
$$
As a consequence, choosing $\ep$ small enough so that $kL_\ep\geq \widetilde L_{\alpha'}$, the comparison principle ensures that \eqref{qqch} holds true, as soon as $\ep>0$ is sufficiently small. This completes the proof of the theorem.
\end{proof}

As in Section 1, introducing for each $\ep\in (0,1]$ the radii $L_\ep^{ext}$ and $L_\ep^{prop}$, from Theorem \ref{th:extinction-deg-mono} and Theorem \ref{th:prop-deg-mono}, we immediately infer the following.

\begin{cor}[Threshold radii for degenerate monostable case] Let Assumption \ref{hyp:deg-mono} hold with $p>1+\frac 2 N$. Then the threshold radii $L_\ep^{ext}\leq L_\ep^{prop}$ for \eqref{eq:deg-mono} satisfy
\begin{equation*}
0<\liminf_{\ep\to 0} \ep^{\frac{p-1}{2}}L_\ep^{ext} \; \text{ and }\; \limsup_{\ep \to 0} \frac{\ep^{\frac{p-1}{2}}}{\left(\ln\frac 1 \ep\right)^{\frac 12}}L_\ep^{prop}<+\infty.
\end{equation*}
\end{cor}

Note that in the one dimensional case, namely $N=1$, it is known that the threshold is sharp, that is $L_\ep^{ext}=L_\ep^{prop}:=L_\ep^\star$ for any $\ep\in (0,1]$. We refer to Du and Matano \cite{Du-Mat-10} for the proof of this sharpness. We also refer to \cite[Theorem 7, Theorem 8, Theorem 9]{Mur-Zho-17} for further results and conditions insuring that the threshold is sharp in the multi-dimensional degenerate monostable case.

As a special case, when $f(u)=ru^p(1-u){\bf 1}_{[0,1]}(u)$ with $p>1+\frac 2 N$, the threshold is sharp (see the proof of Corollary \ref{cor:prototype}), and the above corollary provides rather refined estimates of this quantity in the asymptotic $\ep\to 0$, that can be reformulated as follows.

\begin{cor}[Sharp threshold for prototype degnerate monostable nonlinearity]\label{cor:sharp-mono}
Let $\ep\in (0,1]$ be given. Consider the prototype nonlinearity
$$
f(u)=ru^p(1-u){\bf 1}_{(0,1)}(u), \quad  p>1+\frac 2 N,
$$
where $r>0$.  Then the problem \eqref{eq:deg-mono} exhibits a sharp threshold $L_\ep^\star$ that satisfies
\begin{equation*}
0<\liminf_{\ep\to 0} \ep^{\frac{p-1}{2}}L_\ep^{\star} \; \text{ and }\; \limsup_{\ep \to 0} \frac{\ep^{\frac{p-1}{2}}}{\left(\ln\frac 1 \ep\right)^{\frac 12}}L_\ep^{\star}<+\infty.
\end{equation*}
\end{cor}

 \noindent{\bf Acknowledgements.} %The authors are grateful to the anonymous referee whose precise comments have improved the presentation of the results. 
 M. Alfaro is supported by the 
ANR \textsc{i-site muse}, project \textsc{michel} 170544IA (n$^{\circ}$ ANR \textsc{idex}-0006).


\begin{thebibliography}{99}


\bibitem{Alf-Duc-18} M. Alfaro and A. Ducrot, Population invasion with bistable dynamics and adaptive evolution: the evolutionary rescue, Proc. Amer. Math. Soc. 146 (2018), 4787--4799.

\bibitem{Aro-Wei-75} D. G. Aronson and H. F. Weinberger,  Nonlinear diffusion in population genetics, combustion, and nerve pulse propagation. In Partial differential equations and related topics, Lecture notes in Math. 446, Spriner, Berlin, 1975.

\bibitem{Aro-Wei-78} D. G. Aronson and H. F. Weinberger, Multidimensional nonlinear diffusion arising in population genetics, Adv. Math. 30 (1978), 33--76.

\bibitem{Du-Mat-10} Y. Du and H. Matano, Convergence and sharp thresholds for
propagation in nonlinear diffusion problems, J. Eur. Math. Soc. 12 (2010), 279--312.

\bibitem{FK18} G. Faye and Z.P. Kilpatrick, Threshold of front propagation in neural fields: An interface dynamics approach,
SIAM J. Appl. Math. 78-5 (2018), pp. 2575--2596. 

\bibitem{Fif-Leo-77} P. C. Fife and J. B. McLeod, The approach of solutions of nonlinear diffusion equations to
              travelling front solutions, Arch. Ration. Mech. Anal. 65 (1977), 335--361.
              
\bibitem{flores}
G. Flores, The stable manifold of the standing wave of the Nagumo equation, J. of Differential Equations, 80 (1989), 306--314.
              
\bibitem{Fuj-66} H. Fujita, On the blowing up of solutions of the Cauchy problem for
              $u_{t}=\Delta u+u^{1+\alpha }$, J. Fac. Sci. Univ. Tokyo Sect. I 13 (1966), 109--124.

\bibitem{IB08} I. Idris and V. N. Biktashev, Analytical Approach to Initiation of Propagating Fronts, PRL 101, 244101 (2008).

\bibitem{Kanarek-Webb} A. R. Kanarek and C. T. Webb, Allee effects, adaptive evolution, and invasion success, Evolutionary Applications 3 (2010), 122--135.

\bibitem{Kan-64} Ja. I. Kanel', Stabilization of the solutions of the equations of combustion theory with finite initial functions, Mat. Sb. (N. S.) 65 (107) (1964), 398--413.

\bibitem{McK-Moll}
H.P. McKean and V. Moll, Stabilization of the standing wave in a simple caricature of the nerve equation, Communications on Pure and Applied Mathematics 39, (1986), 485--529.

\bibitem{Moll-Ros}
V. Moll and S.I. Rosencrans, Calculation of the threshold surface for nerve equations, SIAM J. Appl. Math. 50 no 5, (1990), 1419--1441.
              
\bibitem{Mur-Zho-13} C. B. Muratov and X. Zhong, Threshold phenomena for symmetric decreasing solutions of reaction-diffusion equations, NoDEA Nonlinear Differ. Equ. Appl. 20 (2013), 1519--1552. 

\bibitem{Mur-Zho-17} C. B. Muratov and X. Zhong, Threshold phenomena for symmetric-decreasing radial solutions of reaction-diffusion equations, Discrete Contin. Dyn. Syst.  37  (2017), 915--944.

\bibitem{Neu-etal-97} J.C. Neu, R.S. Preissig Jr, W. Krassowska, Initiation of propagation in a one-dimensional excitable medium, Physica D: Nonlinear Phenomena 102(3-4) (1997), 285--299. 

\bibitem{Pol-11} P. Pol\'{a}\v{c}ik, Threshold solutions and sharp transitions for nonautonomous parabolic equations on $\R^N$, Arch. Rational Mech. Anal. 199 (2011), 69--97.

\bibitem{Terman}
D. Terman, A free boundary problem arising from a bistable reaction-diffusion equation, SIAM J. Math. Anal. 14 no 6, (1983), 1107--1129.

\bibitem{Zla-06} A. Zlat\v{o}s,  Sharp transition between extinction and propagation of reaction, J. Amer.
Math. Soc. 19 (2006), 251--263.
\end{thebibliography}
\end{document}